\RequirePackage{amsmath}
\documentclass{article}
\usepackage{bm}
\usepackage[T1]{fontenc}

\usepackage[square,comma,numbers,sort&compress]{natbib}

\newcommand{\RT}{\mathrm{RT}}
\newcommand{\URT}{\mathrm{URT}}
\newcommand{\ART}{\mathrm{ART}}

\renewcommand{\tt}[1]{\textnormal{\texttt{#1}}}

\newcommand{\rev}[1]{{#1}^{R}}

\usepackage{color}

\usepackage{cancel}
\usepackage{tikz}
\usetikzlibrary{trees}

\usepackage{amsthm}

\newtheorem{theorem}{Theorem}
\newtheorem{Lemma}[theorem]{Lemma}

\theoremstyle{definition}
\newtheorem{Conjecture}[theorem]{Conjecture}

\begin{document}
\title{The undirected repetition threshold}
\author{James D.~Currie\thanks
{Supported by the Natural Sciences and Engineering Research Council of Canada (NSERC), [funding reference number 2017-03901].} and Lucas Mol\\
\\
{\normalsize The University of Winnipeg}\\
{\normalsize Winnipeg, Manitoba, Canada R3B 2E9}\\
{\normalsize \{j.currie,l.mol\}@uwinnipeg.ca}}
\date{June 2019}

\maketitle              % typeset the header of the contribution

\begin{abstract}
\noindent
For rational $1<r\leq 2$, an \emph{undirected $r$-power} is a word of the form $xyx'$, where $x$ is nonempty, $x'\in\{x,\rev{x}\}$, and $|xyx'|/|xy|=r$.  The \emph{undirected repetition threshold} for $k$ letters, denoted $\URT(k)$, is the infimum of the set of all $r$ such that undirected $r$-powers are avoidable on $k$ letters.  We first demonstrate that $\URT(3)=\tfrac{7}{4}$.  Then we show that $\URT(k)\geq \tfrac{k-1}{k-2}$ for all $k\geq 4$.  We conjecture that $\URT(k)=\tfrac{k-1}{k-2}$ for all $k\geq 4$, and we confirm this conjecture for $k\in\{4,8,12\}.$

\end{abstract}

\section{Introduction}\label{Intro}

A \emph{square} is a word of the form $xx$, where $x$ is a nonempty word.  An \emph{Abelian square} is a word of the form $x\tilde{x}$, where $\tilde{x}$ is an anagram (or permutation) of $x$.  
% It is well-known that there is an infinite ternary word with no squares as factors~\cite{Berstel1995}, and an infinite $4$-ary word with no Abelian squares as factors~\cite{Keranen1992}.  
The notions of square and Abelian square can be extended to fractional powers in a natural way.  Let $1<r\leq 2$ be a rational number.  An \emph{(ordinary) $r$-power} is a word of the form $xyx$, where $x$ is a nonempty word, and $|xyx|/|xy|=r$.  An \emph{Abelian $r$-power} is a word of the form $xy\tilde{x}$, where $x$ is a nonempty word, $\tilde{x}$ is an anagram of $x$, and $|xy\tilde{x}|/|xy|=r$.\footnote{We use the definition of Abelian $r$-power of Cassaigne and Currie~\cite{CassaigneCurrie1999}.  We note that several distinct definitions exist (see~\cite{SamsonovShur2012,Fici2016}, for example).}

In general, if $\sim$ is an equivalence relation on words that respects length (i.e., we have $|x|=|x'|$ whenever $x\sim x'$), then an \emph{$r$-power up to $\sim$} is a word of the form $xyx'$, where $x$ is nonempty, $x\sim x'$, and $|xyx'|/|xy|=r$. 
%\footnote{Note that we have only defined $r$-powers up to $\sim$ for rational $1<r\leq 2$.  While the definition of ordinary $r$-powers for $r\geq 2$ is standard, the natural definition of $r$-powers up to $\sim$ for $r> 2$ is not entirely clear.}  
The notion of $r$-power up to $\sim$ generalizes ordinary $r$-powers and Abelian $r$-powers, where the equivalence relations are equality and ``is an anagram of'', respectively.

Let $\sim$ be an equivalence relation on words that respects length.  For a real number $1<\alpha\leq 2$, a word $w$ is called \emph{$\alpha$-free up to $\sim$} if no factor of $w$ is an $r$-power up to $\sim$ for $r\geq \alpha$.  Moreover, the word $w$ is called \emph{$\alpha^+$-free up to $\sim$} if no factor of $w$ is an $r$-power up to $\sim$ for $r>\alpha$.  For every integer $k\geq 2$, we say that $\alpha$-powers up to $\sim$ are \emph{$k$-avoidable} if there is an infinite word on $k$ letters that is $\alpha$-free up to $\sim$, and \emph{$k$-unavoidable} otherwise.  For every integer $k\geq 2$, the \emph{repetition threshold up to $\sim$} for $k$ letters, denoted $\RT_\sim(k)$, is defined as
\[
\RT_\sim(k)=\inf\{r\colon\ \mbox{$r$-powers up to $\sim$ are $k$-avoidable}\}.
\]
Since we have only defined $r$-powers for $r\leq 2$, it follows that $\RT_\sim(k)\leq 2$ or $\RT_\sim(k)=\infty$ for any particular value of $k$.

It is well-known that squares are $3$-avoidable~\cite{Berstel1995}.  Thus, for $k\geq 3$, we have that $RT_{=}(k)$ is the usual \emph{repetition threshold}, denoted simply $\RT(k)$.  Dejean~\cite{Dejean1972} proved that $\RT(3)=7/4$, and conjectured that $\RT(4)=7/5$ and $\RT(k)=k/(k-1)$ for all $k\geq 5$.  This conjecture has been confirmed through the work of many authors~\cite{CurrieRampersad2011,Rao2011,Pansiot1984,Dejean1972,Carpi2007,MoulinOllagnier1992,CurrieRampersad2009Again,Noori2007}.

It is also known that Abelian squares are $4$-avoidable~\cite{Keranen1992}.  Let $\approx$ denote the equivalence relation ``is an anagram of''.  Thus, for all $k\geq 4$, we see that $\RT_\approx(k)$ is equal to the \emph{Abelian repetition threshold} (or \emph{commutative repetition threshold}) for $k$ letters, introduced by Cassaigne and Currie~\cite{CassaigneCurrie1999}, and denoted $\ART(k)$.  Relatively less is known about the Abelian repetition threshold.  Cassaigne and Currie~\cite{CassaigneCurrie1999} give (weak) upper bounds on $\ART(k)$ in demonstrating that $\lim_{k\rightarrow\infty}\ART(k)=1$. Samsonov and Shur~\cite{SamsonovShur2012} conjecture that
$\ART(4)=9/5$ and $\ART(k)=(k-2)/(k-3)$ for all $k\geq 5$,
and give a lower bound matching this conjecture.\footnote{Samsonov and Shur define weak, semi-strong, and strong Abelian $\alpha$-powers for all real numbers $\alpha>1$.  For rational $1<r\leq 2$, their definitions of semi-strong Abelian $r$-power and strong Abelian $r$-power are both equivalent to our definition of Abelian $r$-power.}

For every word $x=x_1x_2\cdots x_n$, where the $x_i$ are letters, we let $\rev{x}$ denote the \emph{reversal of $x$}, defined by $\rev{x}=x_n\cdots x_2x_1$.  For example, if $x=\tt{time}$ then $\rev{x}=\tt{emit}$.  Let $\simeq$ be the equivalence relation on words defined by $x\simeq x'$ if $x'=x$ or $x'=\rev{x}$.  In this article, we focus on determining $\RT_{\simeq}(k)$. We simplify our notation and terminology as follows.  We refer to $r$-powers up to $\simeq$ as \emph{undirected $r$-powers}.  These come in two types: words of the form $xyx$ are ordinary $r$-powers, while we refer to words of the from $xyx^R$ as \emph{reverse $r$-powers}.\footnote{We note that words of the form $xyx$ are sometimes referred to as \emph{gapped repeats}, and that words of the form $xy\rev{x}$ are sometimes referred to as \emph{gapped palindromes}.  In particular, an ordinary (reverse, respectively) $r$-power satisfying $r\geq 1+1/\alpha$ is called an \emph{$\alpha$-gapped repeat} (\emph{$\alpha$-gapped palindrome}, respectively).  Algorithmic questions concerning the identification and enumeration of $\alpha$-gapped repeats and palindromes in a given word, along with some related questions, have recently received considerable attention; see~\cite{IKoppl2019, DuchonNicaudPivoteau2018, GawrychowskiManea2015, CrochemoreKolpakovKucherov2016} and the references therein.  Gapped repeats and palindromes are important in the context of DNA and RNA structures, and this has been the primary motivation for their study.}  
For example, the English words \tt{edited} and \tt{render} are undirected $\tfrac{3}{2}$-powers; \tt{edited} is an ordinary $\tfrac{3}{2}$-power, while \tt{render} is a reverse $r$-power.

We say that a word $w$ is \emph{undirected $\alpha$-free} if it is $\alpha$-free up to $\simeq$.  The definition of an \emph{undirected $\alpha^+$-free} word is analogous.  We let $\URT(k)=\RT_{\simeq}(k)$, and refer to this as the \emph{undirected repetition threshold} for $k$ letters.

It is clear that $\simeq$ is coarser than $=$ and finer than $\approx$.  Thus, for every rational $1<r\leq 2$, an $r$-power is an undirected $r$-power, and an undirected $r$-power is an Abelian $r$-power.  As a result, we immediately have
\[
\RT(k)\leq \URT(k)\leq \ART(k)
\]
for all $k\geq 2$.  

Since only a weak upper bound on $\ART(k)$ is currently known, we provide an alternate upper bound on $\URT(k)$ for large enough $k$.  For words $u=u_0u_1\cdots$ and $v=v_0v_1\cdots$ of the same length (possibly infinite) over alphabets $A$ and $B$, respectively, the \textit{direct product} of $u$ and $v,$ denoted $u\otimes v,$ is the word on alphabet $A\times B$ defined by
\[
u\otimes v=(u_0,v_0)(u_1,v_1)\cdots.
\]
A word $x$ is called a \emph{reversible factor} of $w$ if both $x$ and $\rev{x}$ are factors of $w$.

\begin{theorem}\label{123}
For every $k\geq 9$, we have $\URT(k)\leq \RT(\lfloor k/3\rfloor)$.
\end{theorem}

\begin{proof}
Fix $k\geq 9$, and let $\ell=\lfloor k/3\rfloor$.  Evidently, we have $\ell\geq 3$ and thus $\RT(\ell)<2$.  Let $\bm{u}$ be an infinite $\RT(\ell)^+$-free word on $\ell$ letters.  We claim that the word $\bm{u}\otimes (\tt{123})^\omega$ on $3\ell<k$ letters is undirected $\RT(\ell)^+$-free, from which the theorem follows.  Since the only reversible factors of $(123)^\omega$ have length at most $1$, any reverse $r$-power $xy\rev{x}$ in $\bm{u}\otimes (\tt{123})^\omega$ satisfies $x=\rev{x}$, and hence is an ordinary $r$-power as well. Since $\bm{u}$ is ordinary $\RT(\ell)^+$-free, so is $\bm{u}\otimes(\tt{123})^\omega$.  This completes the proof of the claim.
\end{proof}

We now describe the layout of the remainder of the article.  In Section~\ref{Patterns}, we discuss related problems in pattern avoidance, and give some implications of our main results in that setting.  In Section~\ref{URT3}, we show that $\URT(3)=7/4$ using a standard morphic constuction.  In Section~\ref{Backtrack}, we demonstrate that $\URT(k)\geq (k-1)/(k-2)$ for all $k\geq 4$.  In Section~\ref{URT4}, we use a variation of the encoding introduced by Pansiot~\cite{Pansiot1984} to prove that $\URT(k)=(k-1)/(k-2)$ for $k\in\{4,8,12\}$.  In light of our results, we propose the following.

\begin{Conjecture}\label{conj}
For all $k\geq 4$, we have $\URT(k)=(k-1)/(k-2)$.
\end{Conjecture}

We briefly place this conjecture in context. We know that $\RT(k)=k/(k-1)$ for all $k\geq 5,$ we conjecture that $\URT(k)=(k-1)/(k-2)$ for all $k\geq 4$, and Samsonov and Shur~\cite{SamsonovShur2012} conjecture that $\ART(k)=(k-1)/(k-2)$ for all $k\geq 5$.  Let us fix $k\geq 5$.  In~\cite{Shur2011}, Shur proposes splitting all exponents greater than $\RT(k)$ into levels as follows\footnote{Shur considers exponents belonging to the ``extended rationals''.  This set includes all rational numbers and all such numbers with a $+$, where $x^+$ covers $x$, and the inequalities $y\leq x$ and $y<x^+$ are equivalent.}:
\begin{center}
\begin{tabular}{c c c c}
$1$st level & $2$nd level & $3$rd level & $\dots$\\
$\left[\tfrac{k}{k-1}^+,\tfrac{k-1}{k-2}\right]$ & $\left[\tfrac{k-1}{k-2}^+,\tfrac{k-2}{k-3}\right]$ & $\left[\tfrac{k-2}{k-3}^+,\tfrac{k-3}{k-4}\right]$  & $\dots$
\end{tabular}
\end{center}
For $\alpha,\beta\in \left[\tfrac{k}{k-1}^+,\tfrac{k-3}{k-4}\right]$, Shur provides evidence that the language of $\alpha$-free $k$-ary words and the language of $\beta$-free $k$-ary words exhibit similar behaviour (e.g., with respect to growth) if $\alpha$ and $\beta$ are in the same level, and quite different behaviour otherwise; see~\cite{Shur2011,Shur2014}.  If the conjectured values of $\URT(k)$ and $\ART(k)$ are correct, then the undirected repetition threshold and the Abelian repetition threshold provide further evidence of the distinction between levels.

We now introduce some terminology that will be used in the sequel.  Let $A$ and $B$ be alphabets, and let $h\colon A^*\rightarrow B^*$ be a morphism.  Using the standard notation for images of sets, we have $h(A)=\{h(a)\colon\ a\in A\},$ which we refer to as the set of \emph{blocks} of $h$.  
A set of words $P\subseteq A^*$ is called a \emph{prefix code} if no element of $P$ is a prefix of another.
If $P$ is a prefix code and $w$ is a nonempty factor of some element of $P^+$, a \emph{cut} of $w$ over $P$ is a pair $(x,y)$ such that (i) $w=xy$; and (ii) for every pair of words $p,s$ with $pws\in P^+,$ we have $px\in P^*$.  We use vertical bars to denote cuts.  For example, over the prefix code $\{\tt{01},\tt{10}\},$ the word $\tt{11}$ has cut $\tt{1}\vert \tt{1}$.  The prefix code that we work over will always be the set of blocks of a given morphism, and should be clear from context.

\section{Related problems in pattern avoidance}\label{Patterns}

%Let $V$ be a set of letters called variables, and let $\Sigma$ be a finite alphabet.  A word $w\in \Sigma^*$ \emph{encounters} word $p\in V^*$ if $f(p)$ is a factor of $w$ for some non-erasing morphism $f:V^*\rightarrow \Sigma^*$.  Otherwise, $w$ \emph{avoids} $p$.  In this context, $p$ is called a \emph{pattern}.  A pattern $p$ is called \emph{$k$-avoidable} if there are infinitely many words that avoid $p$ on some fixed alphabet of size $k$, and $k$-unavoidable otherwise.  A pattern $p$ is \emph{avoidable} if it is $k$-avoidable for some finite $k$, and \emph{unavoidable} otherwise.
%While there is a nice characterization of avoidable patterns~\cite{BEM1979,Zimin1984}, questions concerning $k$-avoidability are less well understood.  Given pattern $p$ and integer $k\geq 2$, it is unknown whether the $k$-avoidability of $p$ is decidable.  The study of $\alpha$-free and Abelian $\alpha$-free words has implications on the $k$-avoidability of patterns.  For example, it is straightforward to show that any $3/2$-free word $w$ avoids the pattern $p=abacbab$.  Since $\RT(4)=\tfrac{7}{5}$ and $7/5<3/2$, the pattern $p$ is $4$-avoidable.  As another example, since Abelian squares are $4$-avoidable~\cite{Keranen1992}, the pattern $abcacb$ is $4$-avoidable.

Let $p=p_1p_2\cdots p_n$ be a word over alphabet $V$, where the $p_i$ are letters called \emph{variables}.  In this context, the word $p$ is called a \emph{pattern}.  If $\sim$ is an equivalence relation on words, then we say that the word $w$ \emph{encounters $p$ up to $\sim$} if $w$ contains a factor of the form $X_1X_2\cdots X_n$, where each word $X_i$ is nonempty and $X_i\sim X_j$ whenever $p_i=p_j$.  Otherwise, we say that $w$ \emph{avoids $p$ up to $\sim$}.  A pattern $p$ is \emph{$k$-avoidable up to $\sim$} if there is an infinite word on a $k$-letter alphabet that avoids $p$ up to $\sim$.  Otherwise, the pattern $p$ is \emph{$k$-unavoidable up to $\sim$}.  Finally, the pattern $p$ is \emph{avoidable up to $\sim$} if it is $k$-avoidable for some $k$, and \emph{unavoidable up to $\sim$} otherwise.

When $\sim$ is equality, we recover the ordinary notion of \emph{pattern avoidance} (see~\cite{CassaigneChapter}).  When $\sim$ is $\approx$ (i.e., ``is an anagram of''), we recover the notion of \emph{Abelian pattern avoidance} (see~\cite{CurrieLinek2001,CurrieVisentin2008,Rosenfeld2016}, for example).  One could also explore pattern avoidance up to $\simeq$, or \emph{undirected pattern avoidance}. We discuss some initial results in this direction.  While there are patterns that are avoidable in the ordinary sense but not in the Abelian sense~\cite[Lemma 3]{CurrieLinek2001}, every avoidable pattern is in fact avoidable up to $\simeq$, as we show below.

\begin{theorem}\label{Avoidable}
Let $p$ be a pattern.  Then $p$ is avoidable in the ordinary sense if and only if $p$ is avoidable up to $\simeq$.
\end{theorem}

\begin{proof}
If $p$ is unavoidable in the ordinary sense, then clearly $p$ is unavoidable up to $\simeq.$  If $p$ is avoidable in the ordinary sense, then let $\bm{u}$ be an $\omega$-word avoiding $p$.  The direct product $\bm{u}\otimes (\tt{123})^\omega$ avoids $p$ up to $\simeq$ by an argument similar to the one used in Theorem~\ref{123}.
\end{proof}

Questions concerning the $k$-avoidability of patterns up to $\simeq$ appear to be more interesting.  The \emph{avoidability index} of a pattern $p$ up to $\sim$, denoted $\lambda_\sim(p)$, is the least positive integer $k$ such that $p$ is $k$-avoidable up to $\sim$, or $\infty$ if $p$ is unavoidable.  In general, for any pattern $p$, we have
\[
\lambda_{=}(p)\leq \lambda_{\simeq}(p)\leq \lambda_{\approx}(p).
\]
The construction of Theorem~\ref{Avoidable} can be used to show that $\lambda_{\simeq}(p)\leq 3\lambda_{=}(p)$, though we suspect that this bound is not tight.  

The study of the undirected repetition threshold will have immediate implications on avoiding patterns up to $\simeq$.  For example, we can easily resolve the avoidability index of unary patterns up to $\simeq$ using known results along with a result proven later in this article.

\begin{theorem}
$\lambda_\simeq(x^k)=\begin{cases}
3, &\text{if $k\in\{2,3\}$};\\
2, &\text{if $k\geq 4$}.
\end{cases}$
\end{theorem}

\begin{proof}
We prove that $\URT(3)=7/4$ in Section~\ref{URT3}, from which it follows that $\lambda_{\simeq}(xx)=3$.  Backtracking by computer, one finds that the longest binary word avoiding $xxx$ in the undirected sense has length $9$, so $\lambda_{\simeq}(xxx)\geq 3$.  Since $\lambda_{\approx}(xxx)=3$~\cite{Dekking1979}, we conclude that $\lambda_{\simeq}(xxx)=3$.  Finally, since $\lambda_{\approx}(x^4)=2$~\cite{Dekking1979}, we have $\lambda_{\simeq}(x^k)=2$ for all $k\geq 4$.
\end{proof}

\noindent
We plan to determine the avoidability index of all binary patterns up to $\simeq$ in a future work.

%\begin{proposition}\label{Motivation}
%If undirected $4/3$-powers are $k$-avoidable, then the pattern
%\[
%p=abcbadcbabc
%\]
%is $k$-avoidable up to $\simeq$.
%\end{proposition}
%
%\begin{proof}
%Suppose that undirected $4/3$-powers are $k$-avoidable.  It suffices to show that if $w$ is an undirected $4/3$-free word on $k$ letters, then $w$ avoids $p$. Suppose otherwise that $w$ encounters $p$ up to $\simeq$.  Say that $A_1 B_1 C_1 B_2 A_2 D C_2 B_3 A_3 B_4 C_3$ is a factor of $w$, where the $A_i$'s, (respectively $B_i$'s, and $C_i$'s) are equivalent up to $\simeq$.
%
%If $2|B_1|\geq |A_1|$, then $B_3A_3B_4$ is an undirected $r$-power for some $r\geq 4/3$, and this contradicts the fact that $w$ is undirected $4/3$-free.  Similarly, if $2|B_1|\geq |C_1|$, then $B_1C_1B_2$ is an undirected $r$-power for some $r\geq 4/3$.
%
%So we may assume that $2|B_1|<|A_1|,|C_1|$.  We must have either $|A_1|\geq |C_1|$ or $|C_1|>|A_1|$.  However, in the former case, $A_1B_1C_1B_2A_2$ is an undirected $r$-power for some $r\geq 4/3$, while in the latter case, $C_2B_3A_3B_4C_3$ is an undirected $r$-power for some $r>4/3$.  Either way, we have a contradiction.
%\end{proof}

Finally, we remark that the study of $k$-avoidability of patterns up to $\simeq$ has implications for $k$-avoidability of \emph{patterns with reversal} (see~\cite{CurrieLafrance2016,CurrieMolRampersad1,CurrieMolRampersad3} for definitions and examples).  In particular, if pattern $p$ is $k$-avoidable up to $\simeq$, then all patterns with reversal  that are obtained by swapping any number of letters in $p$ with their mirror images are \emph{simultaneously $k$-avoidable}; that is, there is an infinite word on $k$ letters avoiding all such ``decorations'' of $p$.

%For a set of variables $V$, define the \emph{reversed alphabet} $V^R=\{v^R\colon\ v\in V\}$, where $v^R$ denotes the \emph{reversal} or \emph{mirror image} of variable $v$.  A \emph{pattern with reversal} is a word in $(V\cup V^R)^*$.  A morphism $f:(V\cup V^R)^*\rightarrow \Sigma^*$ is said to \emph{respect reversal} if $f(v^R)=(f(v))^R$ for every $v\in V$.  A word $w\in \Sigma^*$ is said to \emph{encounter} pattern with reversal $p$ if $f(p)$ is a factor of $w$ for some non-erasing morphism $f$ that respects reversal.  Otherwise, we say that $w$ \emph{avoids} $p$.  The notions of $k$-avoidability and avoidability index for patterns with reversal are defined as for ordinary patterns.  See~\cite{CurrieLafrance2016,CurrieMolRampersad1,CurrieMolRampersad3} for results on the avoidability index of patterns with reversal.

\section{$\URT(3)=\tfrac{7}{4}$}\label{URT3}

Dejean~\cite{Dejean1972} demonstrated that $\RT(3)=7/4$, and hence we must have $\URT(3)\geq 7/4$.  In order to show that $\URT(3)=7/4,$ it suffices to find an infinite ternary word that is undirected $\tfrac{7}{4}^+$-free.  We provide a morphic construction of such a word.  Let $f$ be the $24$-uniform morphism defined by
\begin{align*}
    \tt{0}&\mapsto \tt{012\,021\,201\,021\,012\,102\,120\,210}\\
    \tt{1}&\mapsto \tt{120\,102\,012\,102\,120\,210\,201\,021}\\
    \tt{2}&\mapsto \tt{201\,210\,120\,210\,201\,021\,012\,102}.
\end{align*}
The morphism $f$ is similar in structure to the morphism of Dejean~\cite{Dejean1972} whose fixed point avoids ordinary $7/4^+$-powers (but not undirected $7/4^+$-powers).  Note, in particular, that $f$ is ``symmetric'' in the sense of~\cite{Frid2001}. 

The following theorem was also verified by one of the anonymous reviewers using the automatic theorem proving software \tt{Walnut}~\cite{Walnut}.

\begin{theorem}
The word $f^\omega(\tt{0})$ is undirected $\tfrac{7}{4}^+$-free.
\end{theorem}

\begin{proof}
We first show that $f^\omega(\tt{0})$ has no factors of the form $xyx^R$ with $|x|>3|y|$ (which is equivalent to $|xy\rev{x}|/|xy|>7/4$).  By exhaustively checking all factors of length $19$ of $f^\omega(0)$, we find that $f^\omega(\tt{0})$ has no reversible factors of length greater than $18$.  So if $f^\omega(\tt{0})$ has a factor of the form $xyx^R$ with $|x|>3|y|$, then $|x|\leq 18$, and in turn $|y|<6$.  So $|xyx^R|<42$.  Every factor of length at most $41$ appears in $f^3(\tt{0})$, so by checking this prefix exhaustively we conclude that $f^\omega(\tt{0})$ has no factors of this form.

So it suffices to show that $f^\omega(\tt{0})$ is (ordinary) $7/4^+$-free.  Suppose towards a contradiction that $f^\omega(\tt{0})$ has factor $xyx$ with $|x|>3|y|$.  Let $n$ be the smallest number such that a factor of this form appears in $f^n(\tt{0})$.  By exhaustive check, we have $n>3$.  First of all, if $|x|\leq 27$, then $|xyx|<63$.  Every factor of $f^\omega(\tt{0})$ of length at most $62$ appears in $f^3(\tt{0})$, so we may assume that $|x|\geq 28$.  Then $x$ contains at least one of the factors $\tt{01020}$, $\tt{12101}$, or $\tt{20212}.$  By inspection, each one of these factors determines a cut in $x$ (over the prefix code $\{f(\tt{0}),f(\tt{1}),f(\tt{2})\}$), say $x=s_x\vert x'\vert p_x$, where $s_x$ is a possibly empty proper suffix of a block of $f$, and $p_x$ is a possibly empty proper prefix of a block of $f$.  If $y$ is properly contained in a single block, then 
\[
xyx=s_x\vert x'\vert p_xys_x\vert x'\vert p_x.
\]
In this case, one verifies that the preimage of $xyx$ contains a square, which contradicts the minimality of $n$.  Otherwise, if $y$ is not properly contained in a single block of $f$, then $y=s_y\vert y'\vert p_y$, where $s_y$ is a possibly empty proper suffix of a block, and $p_y$ is a possibly empty proper prefix of a block.  Then
\[
xyx=s_x\vert x'\vert p_xs_y\vert y'\vert p_ys_x\vert x'\vert p_x,
\]
which appears internally as
\[
\vert p_ys_x\vert x'\vert p_xs_y\vert y'\vert p_ys_x\vert x'\vert p_xs_y\vert.
\]
The preimage of this factor is $ax_1by_1ax_1b,$ where $f(a)=p_ys_x$, $f(b)=p_xs_y$, $f(x_1)=x'$, and $f(y_1)=y'.$  Then $19|ax_1b|\geq |x|>3|y|\geq 3\cdot 19|y_1|,$ or equivalently $|ax_1b|>3|y_1|,$ which contradicts the minimality of $n$.
\end{proof}

\noindent
Thus, we conclude that $\URT(3)=\RT(3)=\frac{7}{4}$.  We will see in the next section that $\URT(k)$ is strictly greater than $\RT(k)$ for every $k\geq 4$.  

\section{A lower bound on $\URT(k)$ for $k\geq 4$}\label{Backtrack}

Here, we prove that $\URT(k)\geq (k-1)/(k-2)$ for $k\geq 4$.
%It follows that $\URT(k)>\RT(k)$ for $k\geq 4$.

\begin{theorem}\label{LowerBoundBacktrack}
If $k\geq 4$, then $\URT(k)\geq \tfrac{k-1}{k-2}$, and the longest $k$-ary word that is undirected $(k-1)/(k-2)$-free has length $k+3$.
\end{theorem}

\begin{proof}
For $k\in\{4,5\}$, the statement is checked by a standard backtracking algorithm, which we performed by computer.  We now provide a general backtracking argument for all $k\geq 6$.

Fix $k\ge 6$, and suppose that $w$ is a $k$-ary word of length $k+4$ that is undirected $(k-1)/(k-2)$-free. It follows that at least $k-2$ letters must appear between any two repeated occurrences of the same letter in $w$, so that any length $k-1$ factor of $w$ must contain $k-1$ distinct letters.  So we may assume that $w$ has prefix $\tt{12}\cdots\tt{(k-1)}$.  Further, given any prefix $u$ of $w$ of length at least $k-1$, there are only two possibilities for the next letter in $w$, as it must be distinct from the $k-2$ distinct letters preceding it.  These possibilities are enumerated in the tree of Figure~\ref{BacktrackingTree}.

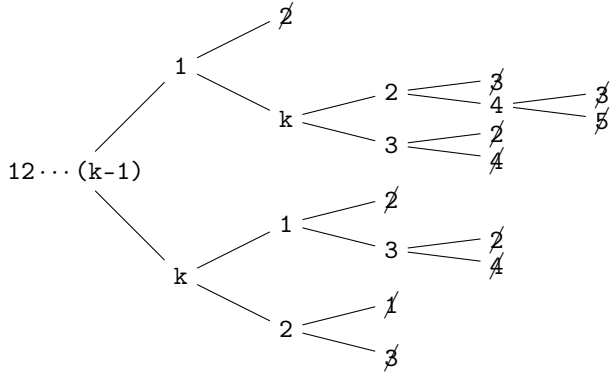
\begin{figure}
    \centering
% Set the overall layout of the tree
\tikzstyle{level 1}=[level distance=2cm, sibling distance=4cm]
\tikzstyle{level 2}=[level distance=2cm, sibling distance=2cm]
\tikzstyle{level 3}=[level distance=2cm, sibling distance=1cm]
\tikzstyle{level 4}=[level distance=2cm, sibling distance=0.5cm]
\tikzstyle{level 5}=[level distance=2cm, sibling distance=0.5cm]

% Define styles for bags and leafs
\tikzstyle{start} = []
\tikzstyle{letter} = []

% The sloped option gives rotated edge labels. Personally
% I find sloped labels a bit difficult to read. Remove the sloped options
% to get horizontal labels. 
\begin{tikzpicture}[grow=right, sloped,scale=0.7]
\node[start,left] {$\tt{12}\cdots\tt{(k-1)}$}
    child {
        node[letter] {\tt{k}}        
            child {
                node[letter] {\tt{2}}
                child {
                    node[letter] {\cancel{\tt{3}}}
                    edge from parent
                }
                child {
                    node[letter] {\cancel{\tt{1}}}
                    edge from parent
                }
                edge from parent
            }
            child {
                node[letter] {\tt{1}}
                child {
                    node[letter] {\tt{3}}
                    child {
                        node[letter] {\cancel{\tt{4}}}
                        edge from parent
                    }
                    child {
                        node[letter] {\cancel{\tt{2}}}
                        edge from parent
                    }
                    edge from parent
                }
                child {
                    node[letter] {\cancel{\tt{2}}}
                    edge from parent
                }
                edge from parent
            }
            edge from parent 
    }
    child {
        node[letter] {\tt{1}}        
        child {
                node[letter] {\tt{k}}
                child {
                    node[letter] {\tt{3}}
                    child {
                        node[letter] {\cancel{\tt{4}}}
                        edge from parent
                    }
                    child {
                        node[letter] {\cancel{\tt{2}}}
                        edge from parent
                    }
                    edge from parent
                }
                child {
                    node[letter] {\tt{2}}
                    child {
                        node[letter] {\tt{4}}
                        child {
                            node[letter] {\cancel{\tt{5}}}
                            edge from parent
                        }
                        child {
                            node[letter] {\cancel{\tt{3}}}
                            edge from parent
                        }
                        edge from parent
                    }
                    child {
                        node[letter] {\cancel{\tt{3}}}
                        edge from parent
                    }
                    edge from parent
                }
                edge from parent
            }
            child {
                node[letter] {\cancel{\tt{2}}}
                edge from parent
            }
        edge from parent         
    };
\end{tikzpicture}
    \caption{The tree of undirected $(k-1)/(k-2)$-power free words on $k$ letters.}
    \label{BacktrackingTree}
\end{figure}

We now explain why each word corresponding to a leaf of the tree contains an undirected $r$-power for some $r\geq (k-1)/(k-2)$.  We examine the leaves from top to bottom, and use the fact that $(k+1)/(k-1)>(k+2)/k>(k-1)/(k-2)$ when $k\geq 6$.
\begin{itemize}
\item The factor $\tt{12}\cdots\tt{(k-1)12}$ is an ordinary $(k+1)/(k-1)$-power.  
\item The factor $\tt{23}\cdots\tt{(k-1)1k23}$ is an ordinary  $(k+2)/k$-power.  
\item The factor $\tt{34}\cdots\tt{(k-1)1k243}$ is a reverse $(k+2)/k$-power.
\item The factor $\tt{45}\cdots\tt{(k-1)1k245}$ is an ordinary  $(k+1)/(k-1)$-power.
\item The factor $\tt{23}\cdots\tt{(k-1)1k32}$ is an ordinary $(k+2)/k$-power.
\item The factor $\tt{34}\cdots\tt{(k-1)1k34}$ is an ordinary $(k+1)/(k-1)$-power.
\item The factor $\tt{12}\cdots\tt{(k-1)k12}$ is an ordinary  $(k+2)/k$-power.
\item The factor $\tt{23}\cdots\tt{(k-1)k132}$ is a reverse $(k+2)/k$-power.
\item The factor $\tt{34}\cdots\tt{(k-1)k134}$ is an ordinary $(k+1)/(k-1)$-power.
\item The factor $\tt{12}\cdots\tt{(k-1)k21}$ is a reverse $(k+2)/k$-power.
\item The factor $\tt{23}\cdots\tt{(k-1)k23}$ is an ordinary $(k+1)/(k-1)$-power. \qedhere
\end{itemize}
\end{proof}

Conjecture~\ref{conj} proposes that the value of $\URT(k)$ matches the lower bound of Theorem~\ref{LowerBoundBacktrack} for all $k\geq 4$.  In the next section, we confirm Conjecture~\ref{conj} for several values of $k$.

\section{$\URT(k)=\tfrac{k-1}{k-2}$ for $k\in\{4,8,12\}$}\label{URT4}

First we explain why we rely on a different technique than in Section~\ref{URT3}.  Fix $k\geq 4$, and let $\Sigma_k=\{\tt{1},\tt{2},\dots,\tt{k}\}$. A morphism $h:A^*\rightarrow B^*$ is called $\alpha$-free ($\alpha^+$-free, respectively) if it maps every $\alpha$-free ($\alpha^+$-free, respectively) word in $A^*$ to an $\alpha$-free ($\alpha^+$-free, respectively) word in $B^*$.  The morphism $h$ is called \emph{growing} if $h(a)>1$ for all $a\in A^*$.  Brandenburg~\cite{Brandenburg1983} demonstrated that for every $k\geq 4$, there is no growing $\RT(k)^+$-free morphism from $\Sigma_k^*$ to $\Sigma_k^*$.  By a minor modification of his proof, one can show that there is no growing $(k-1)/(k-2)^+$-free morphism from $\Sigma_k^*$ to $\Sigma_k^*$.  While this does not entirely rule out the possibility that there is a morphism from $\Sigma_k^*$ to $\Sigma_k^*$ whose fixed point is $(k-1)/(k-2)^+$-free, it suggests that different techniques may be required.  Our technique relies on an encoding similar to the one introduced by Pansiot~\cite{Pansiot1984} in showing that $\RT(4)=7/5$.  Pansiot's encoding was later used in all subsequent work on Dejean's Conjecture.

\subsection{A ternary encoding}

We first describe an alternate definition of ordinary $r$-powers which will be useful in this section.  A word $w=w_1\cdots w_n$, where the $w_i$ are letters, is \emph{periodic} if for some positive integer $q$, we have $w_{i+q}=w_i$ for all $1\leq i\leq n-q$.  In this case, the integer $q$ is called a \emph{period} of $w$.  The \emph{exponent} of $w$, denoted $\exp(w),$ is the ratio between its length and its minimal period.  If $r=\exp(w)$, then $w$ is an \emph{$r$-power}.\footnote{If $r\leq 2$, then $w$ is an $r$-power as we have defined it in Section~\ref{Intro}. If $r>2$, then we take this as the definition of an (ordinary) $r$-power. For example, the English word \tt{alfalfa} has minimal period $3$ and exponent $\tfrac{7}{3},$ so it is a $\tfrac{7}{3}$-power.}  We can write any $r$-power $w$ as $w=pe$, where $|pe|/|p|=r$ and $e$ is a prefix of $pe$.  In this case, we say that $e$ is the \emph{excess} of the $r$-power $w$.

Suppose that $w\in \Sigma_k^*$ is an undirected $(k-1)/(k-2)^+$-free word that contains at least $k-1$ distinct letters.  Write $w=w_1w_2\cdots w_n$ with $w_i\in \Sigma_k$.  Certainly, every length $k-2$ factor of $w$ contains $k-2$ distinct letters, and it is easily checked that every length $k$ factor of $w$ contains at least $k-1$ distinct letters.

Now let $w\in \Sigma_k^*$ be any word containing at least $k-1$ distinct letters and satisfying these two properties:
\begin{itemize}
    \item Every length $k-2$ factor of $w$ contains $k-2$ distinct letters; and
    \item Every length $k$ factor of $w$ contains at least $k-1$ distinct letters.
\end{itemize}
Let $u$ be the shortest prefix of $w$ containing $k-1$ distinct letters.  We see immediately that $u$ has length $k-1$ or $k$.  Write $w=uv$, where $v=v_1v_2\cdots v_{n}$ with $v_i\in \Sigma_k$.  Define $p_0=u$ and $p_i=uv_1\cdots v_i$ for all $i\in\{1,\dots,n\}$.  For all $i\in\{0,1,\dots,n\}$, the prefix $p_i$ determines a permutation
\[
r_i=\begin{pmatrix}
\tt{1} & \tt{2} & \dots & \tt{k}\\ 
r_i[\tt{1}] & r_i[\tt{2}] & \dots & r_i[\tt{k}]
\end{pmatrix},
\]
of the letters of $\Sigma_k$, which ranks the letters of $\Sigma_k$ by the index of their final appearance in $p_i$.  In other words, the word $r_i[\tt{3}] \cdots r_i[\tt{k}]$ is the length $k-2$ suffix of $p_i$, and of the two letters in $\Sigma_k\backslash\{r_i[3],\dots,r_i[k]\}$, the letter $r_i[\tt{2}]$ is the one that appears last in $p_i$. Note that the final letter $r_i[\tt{1}]$ may not even appear in $p_i$.  For example, on $\Sigma_6,$ the prefix \tt{123416} gives rise to the permutation
\[
\begin{pmatrix}
\tt{1} & \tt{2} & \tt{3} & \tt{4} & \tt{5} & \tt{6}\\
\tt{5} & \tt{2} & \tt{3} & \tt{4} & \tt{1} & \tt{6}
\end{pmatrix}.
\]
Since every factor of length $k-2$ in $w$ contains $k-2$ distinct letters, for any $i\in\{1,\dots,n\}$, the letter $v_i$ must belong to the set $\{r_{i-1}[1],r_{i-1}[2],r_{i-1}[3]\}.$  This allows us to encode the word $w$ over a ternary alphabet, as described explicitly below.

For $1\leq i\leq n$, define $t(w)=t_1\cdots t_{n}$, where for all $1\leq i\leq n$, we have
\[
t_i=\begin{cases}
\tt{1}, & \text{if $w_i=r_{i-1}[1]$};\\
\tt{2}, & \text{if $w_i=r_{i-1}[2]$};\\
\tt{3}, & \text{if $w_i=r_{i-1}[3]$}.
\end{cases}
\]

\noindent
For example, on $\Sigma_5$, for the word $w=\tt{12342541243},$
the shortest prefix containing $4$ distinct letters is $1234$, and $w$ has encoding $t(w)=\tt{3131231}.$
Given the shortest prefix of $w$ containing $k-1$ distinct letters, and the encoding $t(w)$, we can recover $w$.  Moreover, if $w$ has period $q<n$, then so does $t(w)$.  The exponent $|w|/q$ of $w$ corresponds to an exponent $|v|/q$ of $t(w)$.

Let $S_k$ denote the symmetric group on $\Sigma_k$ with left multiplication.  Define a morphism $\sigma:\Sigma_3^*\rightarrow S_k$ by
\begin{align*}
\sigma(1)&=\begin{pmatrix}
    1 & 2 & 3 & 4 & \dots & k-1 & k \\
    2 & 3 & 4 & 5 & \dots & k & 1
  \end{pmatrix}\\
\sigma(2)&=\begin{pmatrix}
    1 & 2 & 3 & 4 & \dots & k-1 & k \\
    1 & 3 & 4 & 5 & \dots & k & 2
  \end{pmatrix}\\
\sigma(3)&=\begin{pmatrix}
    1 & 2 & 3 & 4 & \dots & k-1 & k \\
    1 & 2 & 4 & 5 & \dots & k & 3
  \end{pmatrix}.
\end{align*}
One proves by induction that $r_0\sigma(t(p_i))=r_i$.  It follows that if $w=pe$ has period $|p|$, and $e$ contains at least $k-1$ distinct letters, then the length $|p|$ prefix of $t(w)$ lies in the kernel of $\sigma$.  In this case, the word $t(w)$ is called a \emph{kernel repetition}. For example, over $\Sigma_4$, the word
\[
w=\tt{123243414212324}
\]
has period $10$, and excess $\tt{12324}$.  Hence, the encoding $t(w)=\tt{312313123131}$ is a kernel repetition; one verifies that $\sigma(\tt{3123131231})=\mbox{id}$.

Suppose that $k$ is even.  Then $\sigma(1)$ and $\sigma(3)$ are odd, while $\sigma(2)$ is even.  It follows that $\sigma(31)$ is even, and hence the subgroup of $S_k$ generated by $\sigma(2)$ and $\sigma(31)$ is a subgroup of the alternating group $A_k$.  This simple observation leads to the following important lemma, which will be used to bound the length of reversible factors in the words we construct.

\begin{Lemma}\label{Alternating}
Let $k\geq 4$ satisfy $k\equiv 0\pmod{4}$.  Let $w\in \Sigma_k^*$ be a word with prefix $\tt{12}\cdots\tt{(k-1)}$ and encoding $t(w)\in\{\tt{31},\tt{2}\}^*$.  Suppose that $u=u_1u_2\cdots u_{k-1}$ is a factor of $w$, where $u_1,u_2,\dots,u_{k-1}\in \Sigma_k$ are distinct letters.  Then $u^R$ is not a factor of $w$.
\end{Lemma}
\begin{proof}
Suppose towards a contradiction that $u$ and $u^R$ are both factors of $w$.  Assume without loss of generality that $u$ appears before $u^R$ in $w$.  Then $w$ contains a factor $x$ with prefix $u$ and suffix $u^R$.  Consider the encoding $t(x),$ which is a factor of $t(w)$.

Immediately after reading $u$, the ranking of the letters in $\Sigma_k$ is
\[
\begin{pmatrix}
\tt{1} & \tt{2} & \tt{3} & \dots & \tt{k}\\ 
u_k & u_1 & u_2 & \dots & u_{k-1}
\end{pmatrix},
\]
where $u_k$ is the unique letter in $\Sigma_k\backslash\{u_1,u_2,\dots,u_{k-1}\}$.
Immediately after reading $u^R$, the ranking of the letters in $\Sigma_k$ is
\[
\begin{pmatrix}
\tt{1} & \tt{2} & \tt{3} & \dots & \tt{k}\\ 
u_k & u_{k-1} & u_{k-2} & \dots & u_1
\end{pmatrix}.
\]
Evidently, we have 
\[
\sigma(t(x))=\begin{pmatrix}
\tt{1} &\tt{2} & \tt{3} & \dots & \tt{k-1} & \tt{k}\\
\tt{1} & \tt{k} & \tt{k-1} & \dots & \tt{3} & \tt{2}
\end{pmatrix}.
\]

Since $k\equiv 0\pmod{4},$ we observe that $\sigma(t(x))$ is an odd permutation.  We claim that $t(x)$ does not begin in $\tt{1}$ or end in $\tt{3}$, so that $t(x)\in\{\tt{31},\tt{2}\}^*$.  But $\sigma(\tt{31})$ and $\sigma(\tt{2})$ are both even, which contradicts the fact that $\sigma(t(x))$ is odd.

The fact that $t(x)$ does not end in $\tt{3}$ follows immediately from the fact that $u_{k-1}\neq u_1$.  It remains to show that $t(x)$ does not begin with $\tt{1}$.  If $x$ is a prefix of $w$, then $t(x)$ begins in $\tt{3}$ or $\tt{2}$, so we may assume that $w=yxz$ with $y\neq \varepsilon$.  Then $t(w)$ has prefix $t(yu)t(x)$.  If $t(x)$ began in $\tt{1}$, then $t(yu)$ would necessarily end in $3$, and this is impossible since $u_1\neq u_{k-1}$.  This completes the proof of the claim, and the lemma.
\end{proof}

% \noindent
% Note that
% \[
% \sigma(31)=\begin{cases}\begin{pmatrix}
% 1 & 4 & 2
% \end{pmatrix} \mbox{ if } k=4;\\
% \begin{pmatrix}
% 1 & k & k-2 & \dots & 2 
% \end{pmatrix}
% \begin{pmatrix}
% 3 & k-1 & k-3 & \dots & 5
% \end{pmatrix} \mbox{ if $k>4$, $k$ even};\\
% \begin{pmatrix}
% 1 & k & k-2 & \dots & 3
% \end{pmatrix}
% \begin{pmatrix}
% 2 & k-1 & k-3 & \dots & 4
% \end{pmatrix} \mbox{ if  $k>4$, $k$ odd}.
% \end{cases}
% \]

\subsection{Constructions}

Define morphisms $f_4,f_8,f_{12}:\Sigma_2^*\rightarrow \Sigma_2^*$ as follows:
\begin{align*}
f_4(\tt{1})&=\tt{121}\\
f_4(\tt{2})&=\tt{122}\\[5pt]
f_8(\tt{1})&=\tt{121212112122121}\\
f_8(\tt{2})&=\tt{211212122122112}\\[5pt]
f_{12}(\tt{1})&=\tt{121212121211212122121}\\
f_{12}(\tt{2})&=\tt{212122112121121212212}.
\end{align*}  
Define $g:\Sigma_2^*\rightarrow \Sigma_3^*$ by
\begin{align*}
g(\tt{1})&=\tt{31}\\
g(\tt{2})&=\tt{312}.
\end{align*}
A key property of each of the morphisms $f_4,$ $f_8,$ $f_{12},$ and $g$ is that the images of $\tt{1}$ and $\tt{2}$ end in different letters.  
%For $h\in\{f_4,f_8,f_{12},g\}$, let $u$ be a factor of some word $v\in\{h(\tt{1}),h(\tt{2})\}^*$.  If $u$ has a cut, then we can write $u=s|u'|p$ where $s$ is a proper suffix of a block of $h$ and $p$ is a proper prefix of a block of $h$.  Moreover, if $s\neq \varepsilon$, then there is a unique block $h(\tt{a})$ ending in $s$, i.e., $u$ must appear internally as $h(\tt{a})|u'|p$.  This property was referred to as ``markability'' in~\cite{CurrieRampersad2011}.

\begin{theorem}\label{Main}
Fix $k\in\{4,8,12\},$ and let $f=f_k$.  Let $\bm{w}$ be the word over $\Sigma_k$ with prefix $\tt{12}\cdots\tt{(k-1)}$ and encoding $g(f^\omega(\tt{1}))$.  Then $\bm{w}$ is undirected $(k-1)/(k-2)^+$-free.
\end{theorem}

The remainder of this section is devoted to proving Theorem~\ref{Main}.  Essentially, we adapt the technique first used by Moulin-Ollagnier~\cite{MoulinOllagnier1992}.  A simplified version of Moulin-Ollagnier's technique, which we follow fairly closely, is exhibited by Currie and Rampersad~\cite{CurrieRampersad2011}.  For the remainder of this section, we use notation as in Theorem~\ref{Main}.  We let $r=|f(\tt{1})|$, i.e., we say that $f$ is $r$-uniform.

We first discuss kernel repetitions appearing in $g(f^\omega(\tt{1}))$.
% Let $h:A\rightarrow B$ be a morphism and let $C_h=\{h(\tt{a})\colon \ a\in A\}$.  A word $v\in B^*$ is \emph{markable} with respect to $h$ if whenever $h(X)xv$ and $h(Y)yv$ are prefixes of $h(w)$ for some word $w\in A^*$, with $x$ and $y$ proper prefixes of words in $C_h,$ then $x=y$. 
Let factor $v=pe$ of $g(f^\omega(\tt{1}))$ be a kernel repetition with period $q$; say $g(f^\omega(\tt{1}))=xv\bm{y}$.  Let $V=x'vy'$ be the maximal period $q$ extension of the occurrence $xv\bm{y}$ of $v$.  Write $x=Xx'$ and $\bm{y}=y'\bm{Y},$ so that $g(f^\omega(\tt{1}))=XV\bm{Y}$.  Write $V=PE=EP',$ where $|P|=q$.  By the periodicity of $PE$, the factor $P$ is conjugate to $p$, and hence $P$ is in the kernel of $\sigma$.  Write $P=\pi''g(\pi)\pi'$ where $\pi''$ is a proper suffix of $g(\tt{1})$ or $g(\tt{2})$, and $\pi'$ is a prefix of $g(\tt{1})$ or $g(\tt{2})$.  Analogously, write $E=\eta''g(\eta)\eta'$.  Since $g(\tt{1})$ and $g(\tt{2})$ end in different letters, it follows from the maximality of $V$ that $\pi''=\eta''=\varepsilon$.  In particular, the word $P$ begins in $\tt{3}$.  It follows that $E$ begins in $\tt{3}$, and thus we may assume that $\pi'=\varepsilon$.  Finally, by the maximality of $V$, we have $\eta'=\tt{31}$, the longest common prefix of $g(\tt{1})$ and $g(\tt{2})$.  Altogether, we can write
\[
PE=g(\pi\eta)\tt{31},
\]
where $g(\pi)=P$ and $\eta$ is a prefix of $\pi$.  We see that $|P|\geq 2|\pi|$ and $|E|\leq 3|\eta|+2$.

Let $\tau:\Sigma_2^*\rightarrow S_k$ be the composite morphism $\sigma\circ g$.  Evidently, we have
\begin{align*}
\tau(\tt{1})&=\sigma(g(\tt{1}))=\sigma(\tt{31})%=\begin{pmatrix} 1 & 4 & 2\end{pmatrix}
, \mbox{ and}\\
\tau(\tt{2})&=\sigma(g(\tt{2}))=\sigma(\tt{312})%=\begin{pmatrix} 1 & 3 & 2\end{pmatrix}
.
\end{align*}
Since $P$ was in the kernel of $\sigma$, we see that
\[
\tau(\pi)=\sigma(g(\pi))=\sigma(P)=\mathrm{id},
\]
i.e., the word $\pi$ is in the kernel of $\tau$.

Now set $\pi_0=\pi$ and $\eta_0=\eta$.  By the maximality of $PE,$ the repetition $\pi\eta=\pi_0\eta_0$ must be a maximal repetition with period $|\pi_0|$ (i.e., it cannot be extended).  If $\eta_0$ has a cut, then it follows by arguments similar to those used above that $\pi_0\eta_0=f(\pi_1\eta_1)\eta'$, where $\eta_1$ is a prefix of $\pi_1$ and $\eta'$ is the longest common prefix of $f(\tt{1})$ and $f(\tt{2})$.  One checks that there is an element $\phi\in S_k$ such that
\begin{align*}
\phi\cdot \tau(f(\tt{a}))\cdot \phi^{-1}=\tau(\tt{a})
\end{align*}
% \begin{align*}
% \begin{pmatrix} 2 & 3 & 4 \end{pmatrix}\cdot \tau(f(\tt{a}))\cdot \begin{pmatrix} 2 & 3 & 4 \end{pmatrix}^{-1}=\tau(\tt{a})
% \end{align*}
for every $\tt{a}\in\{\tt{1},\tt{2}\}$, i.e., the morphism $\tau$ satisfies the ``algebraic property'' described by Moulin-Ollagnier~\cite{MoulinOllagnier1992}.  It follows that $\pi_1$ is in the kernel of $\tau$.  We can repeat this process until we reach a repetition $\pi_s\eta_s$ whose excess $\eta_s$ has no cut.  Recalling that $f$ is an $r$-uniform morphism, we have
\[
|\pi_0|=r^s|\pi_s|
\]
and 
\[
|\eta_0|=r^s|\eta_s|+|\eta'|\sum_{i=0}^{s-1}r^i%=r^s|\eta_s|+r^s-1
.
\]
Note that $|\eta'|=2$ if $k=4$, while $|\eta'|=0$ if $k\in\{8,12\}$.  Thus, we have
\begin{align*}
|\eta_0|&=
\begin{cases}
r^s|\eta_s|+r^s-1, &\text{if $k=4$};\\
r^s|\eta_s|, &\text{if $k\in\{8,12\}$}.
\end{cases}
\end{align*}
It follows that $|\eta_0|\leq r^s|\eta_s|+r^s-1$.

\begin{proof}[Proof of Theorem~\ref{Main}]
We first show that $\bm{w}$ contains no reverse $\alpha$-power with $\alpha>(k-1)/(k-2)$.  Since $\tt{33}$ is not a factor of $g(f^\omega(\tt{1}))$, every factor of length $k$ in $\bm{w}$ contains a factor of the form $u=u_1u_2\cdots u_{k-1}$, where $u_1,u_2,\dots,u_{k-1}$ are distinct letters.  Thus, by Lemma~\ref{Alternating}, if $xy\rev{x}$ is a factor of $\bm{w}$ with $|xy\rev{x}|/|xy|>(k-1)/(k-2)$, then $|x|\leq k-1$.  In turn, we have $|xy\rev{x}|< (k-1)^2$.  Therefore, we conclude by a finite check that $\bm{w}$ contains no reverse $\alpha$-power with $\alpha>(k-1)/(k-2)$. 

It remains to show that $\bm{w}$ is ordinary $(k-1)/(k-2)^+$-free.  Suppose to the contrary that $pe$ is a factor of $\bm{w}$ such that $e$ is a prefix of $pe$ and $|pe|/|p|>(k-1)/(k-2)$.  We may assume that $pe$ is maximal with respect to having period $|p|$.  If $e$ has less than $k-1$ distinct letters, then $|e|\leq k-1$.  In turn, we have $|pe|<(k-1)^2$.  By a finite check, the word $\bm{w}$ has no such factors.

So we may assume that $e$ has at least $k-1$ distinct letters.  Let $V=t(pe)$, and let $P$ be the length $|P|$ prefix of $V$.  So $V=PE$, where $E$ is a prefix of $P$.  Hence $V$ is a kernel repetition, i.e., the word $P$ is in the kernel of $\sigma$.  By the maximality of $pe$, we see that $P$ begins in $\tt{3}$.  Hence, the length $k-1$ prefix of $p$ contains $k-1$ distinct letters, and $|e|=|E|+k-1$.  We can find a factor $\pi_s\eta_s$ of $f^\omega(\tt{1})$ as described above, such that $\eta_s$ is a prefix of $\pi_s\eta_s$, the word $\pi_s$ is in the kernel of $\tau$, and $\eta_s$ does not contain a cut.  Now
\begin{align*}
    \frac{1}{k-2}&< \frac{|e|}{|p|}\\
    &=\frac{|E|+k-1}{|P|}\\
    &\leq \frac{3|\eta_0|+k+1}{2|\pi_0|}\\
    &=\frac{3\left(r^s|\eta_s|+r^s-1\right)+k+1}{2r^s|\pi_s|}\\
    &=\frac{3r^s\left(|\eta_s|+1\right)+k-2}{2\cdot r^s|\pi_s|}\\
    &=\frac{3(|\eta_s|+1)+(k-2)r^{-s}}{2|\pi_s|}\\
    &\leq \frac{3|\eta_s|+k+1}{2|\pi_s|}.
\end{align*}
Thus, we have 
\begin{align}\label{General}
|\pi_s|<\frac{(k-2)(3|\eta_s|+k+1)}{2}.
\end{align}

By exhaustive check, every factor of length $r$ in $f^\omega(\tt{1})$ contains a cut, so we must have $|\eta_s|<r$, and we can list all possibilities for $\eta_s$.  For each possible value of $\eta_s$, we can enumerate all possibilities for $\pi_s$ using~(\ref{General}).  At this point, our argument depends on the value of $k$.  

If $k\in\{8,12\}$, then we find that no such factor $\pi_s\eta_s$ exists in $f^\omega(\tt{1})$.  On the other hand, if $k=4$, then we find only the following two pairs satisfying~(\ref{General}):
\begin{itemize}
\item $\pi_s=\tt{2121} \mbox{ and } \eta_s=\varepsilon;$
\item
$\pi_s=\tt{2112112212} \mbox{ and } \eta_s=\tt{21}.$
\end{itemize}
Note, however, that for each pair, we have $|\pi_s|_\tt{1}=|\pi_s|_\tt{2}$ and $|\eta_s|_\tt{1}=|\eta_s|_\tt{2}$.  Since $f_4(\tt{1})=\tt{121}$ and $f_4(\tt{2})=\tt{122},$ it follows that $|\pi_0|_\tt{1}=|\pi_0|_\tt{2}$ and $|\eta_0|_\tt{1}=|\eta_0|_\tt{2}$.  In this case, we have $|P|=\tfrac{5}{2}|\pi_0|$ and $|E|=\tfrac{5}{2}|\eta_0|+2$.  By adapting the string of inequalities leading to~(\ref{General}), we find that we must in fact have
\begin{align*}%\label{HalfAndHalf}
|\pi_s|<2|\eta_s|+4,
\end{align*}
and this does not hold in either case.

Thus, we conclude that $\bm{w}$ is undirected $(k-1)/(k-2)^+$-free.
\end{proof}

\section*{Acknowledgements}

We thank the anonymous reviewers, whose comments helped to improve the article.

%\bibliographystyle{splncsnat}
%\bibliography{/Users/lgmol/Documents/RT}

\end{document}